\documentclass[letterpaper, 10 pt, conference]{ieeeconf} 

\IEEEoverridecommandlockouts
\usepackage{times}
\usepackage[labelformat=simple]{subcaption}
\usepackage[utf8]{inputenc} 
\usepackage[T1]{fontenc}    
\usepackage{url}            
\usepackage{booktabs}       
\usepackage{amsfonts}       
\usepackage{nicefrac}       
\usepackage{microtype}      
\usepackage{balance}      
\usepackage[font=footnotesize]{caption}
\usepackage{wrapfig}
\usepackage{bibunits}
\usepackage{graphics}

\usepackage{graphicx,caption,subcaption,xspace}
\usepackage{amsmath,amssymb,stmaryrd,mathtools}
\usepackage{algorithm}
\usepackage{algpseudocode}
\usepackage{float}
\usepackage{acronym}
\usepackage{comment}
\usepackage{color}
\usepackage{units}
\usepackage{textcomp}

\usepackage{multirow}
\usepackage[capitalise]{cleveref}

\usepackage{tikz}

\usepackage{cite}
\usepackage{url}

\usepackage{booktabs}

\crefname{equation}{}{} 
\crefname{section}{Sec.}{Sec.}

\newcommand{\sinit}{s_{0}}
\newcommand{\wifinal}{W_{i}(x_{i},0)} 
\newcommand{\wis}{W_{i}(x_{i},s)} 
\newcommand{\wiinit}{W_{i}(x_{i},s_{0})} 
\newcommand{\si}{S_{i}} 
\newcommand{\sic}{S_{i}^{c}} 
\newcommand{\xii}{x_{i}} 

\newcommand{\R}{\mathbb{R}} 

\newcommand{\argmax}{\operatornamewithlimits{argmax}}

\newtheorem{theorem}{Theorem}

\algdef{SE}[SUBALG]{Indent}{EndIndent}{}{\algorithmicend\ }%
\algtext*{Indent}
\algtext*{EndIndent}

\title{\large \bf
Guaranteed-Safe Approximate Reachability via State Dependency-Based Decomposition
}

\author{Anjian Li$^1$, Mo Chen$^1$
\thanks{$^1$School of Computing Science, Simon Fraser University, Burnaby, BC, Canada, \{anjianl, mochen\}@sfu.ca
}
}

\begin{document}

\maketitle
\thispagestyle{empty}
\pagestyle{empty}

\begin{abstract}
Hamilton Jacobi (HJ) Reachability is a formal verification tool widely used in robotic safety analysis. 
Given a target set as unsafe states, a dynamical system is guaranteed not to enter the target under the worst-case disturbance if it avoids the Backward Reachable Tube (BRT).
However, computing BRTs suffers from exponential computational time and space complexity with respect to the state dimension. 
Previously, system decomposition and projection techniques have been investigated, but the trade off between applicability to a wider class of dynamics and degree of conservatism has been challenging. 
In this paper, we propose a State Dependency Graph to represent the system dynamics, and decompose the full system where only dependent states are included in each subsystem, and ``missing'' states are treated as bounded disturbance.
Thus for a large variety of dynamics in robotics, BRTs can be quickly approximated in lower-dimensional chained subsystems with the guaranteed-safety property preserved.
We demonstrate our method with numerical experiments on the 4D Quadruple Integrator, and the 6D Bicycle, an important car model that was formerly intractable.
\end{abstract}

\section{Introduction} \label{sec:intro}

As the popularity of mobile autonomous systems rapidly grows in daily life, the importance of safety of these systems substantially increases as well. 
Especially formal verification is urgently needed for safe-critical systems like self-driving cars, drones, and etc., where any crash causes serious damage.

Optimal control and differential game theory are well-studied for controlled nonlinear systems experiencing adversarial disturbances \cite{barron1990differential,mitchell2005time,margellos2011hamilton,bokanowski2011minimal}.
Reachability analysis is a powerful tool to characterize the safe states of the systems and provides safety controllers \cite{kurzhanski2002ellipsoidal, frehse2011spaceex,althoff2013reachability}. 
It has been widely used in trajectory planning \cite{parzani2018hamilton, herbert2017fastrack,singh2018robust,landry2018reach}, air traffic management \cite{chen2015safe,chen2017reachability}, and multi-agent collision avoidance \cite{chen2016multi,dhinakaran2017hybrid}.
In a collision avoidance scenario, the Backward Reachable Tube (BRT) represents the states from which reaching the unsafe states is inevitable within a specified time horizon under the worst-case disturbances \cite{chen2018hamilton}.

There is a variety of reachability analysis methods. 
\cite{frehse2011spaceex, althoff2010computing} focus on analytic solutions, which are fast to compute but require specific types of targets, e.g. polytopes or hyperplanes. 
Some other techniques have strong assumptions such as linear dynamics \cite{kurzhanski2002ellipsoidal,maidens2013lagrangian}, or dynamics that do not include any control and disturbance \cite{darbon2016algorithms}.
HJ reachability is the most flexible method that accommodates nonlinear dynamics and arbitrary shapes of target sets; however, such flexibility requires level set methods \cite{mitchell2000level, osher2004level}, in which value functions are stored on grid points in the discretized state space. 
Such an approach suffers from the curse of dimensionality.

\begin{figure}[htbp]
\begin{subfigure}[b]{.4\linewidth}
\centering
\parbox[][1.8cm][c]{\linewidth}{
\includegraphics[scale=0.35]{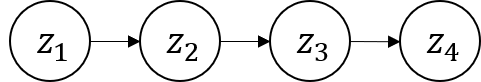}}
\caption{}
\label{fig:sdg_4dint}
\end{subfigure}
%
%
\begin{subfigure}[b]{.6\linewidth}
\centering
\includegraphics[scale=0.35]{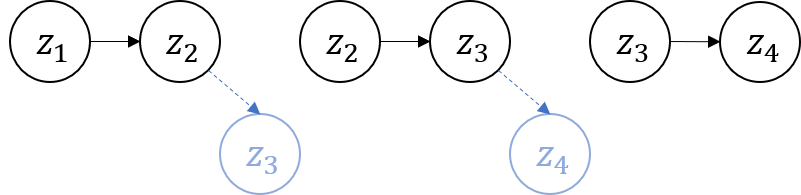}
\caption{}
\label{fig:decompose_sdg_4dint}
\end{subfigure}
\caption{The State Dependency Graph \ref{fig:sdg_4dint} and the decomposed State Dependency Graph \ref{fig:decompose_sdg_4dint} for 4D Integrator. The blue vertices and edges indicate the missing states and their dependencies. Our method approximates the BRT of the full-dimensional system in Fig. \ref{fig:sdg_4dint} by concurrently computing a sequence of BRTs for the subsystems in Fig. \ref{fig:decompose_sdg_4dint}.}
\vspace{-2em}
\end{figure}


Several approaches have been proposed to reduce the computation burden for HJ reachability. 
Projection methods were used to approximate BRTs in lower-dimension spaces \cite{mitchell2003overapproximating}, but the results could be overly conservative at times. 
Integrator structures were analyzed in \cite{mitchell2011scalable} to reduce dimensionalilty by one.
The state decoupling disturbances method \cite{chen2016fast} treated certain states as disturbance, thus systems can be decoupled to lower dimension and computed for goal-reaching problems.
An exact system decomposition method \cite{chen2018decomposition} could reduce computation burden without incurring approximation errors, but  was only applicable if self-contained systems existed.

In this paper, we propose a novel system decomposition method that exploits state dependency information in the system dynamics in a more sophisticated, multi-layered manner compared to previous works. 
Our approach represents the system dynamics using a directed dependency graph, and decompose the full system into subsystems based on this graph, which allows BRT over-approximations to be tractable yet not overly conservative.
``Missing'' states in each subsystem are treated as disturbances bounded by other concurrently computed BRT approximations.

Our method is applicable to a large variety of system dynamics, especially those with a loosely coupled, ``chained'' structure.
Combining our method with other decomposition techniques could achieve more dimensionality reduction.
Beyond that, we also offers a flexible way of adjusting the trade-off between computational complexity and degree of conservatism, which is adaptable to the amount of computational resources available.


\textbf{Organization}: 
In Section \ref{sec:background}, we introduce the background on HJ reachability and projection operations.
In Section \ref{sec:chain_model}, we present state dependency graph and the novel decomposition method for dynamical systems to compute BRT over-approximations.
Proof of correctness and other discussion are also offered.
In Section \ref{sec:exp}, we present numerical results for the 4D Quadruple Integrator and 6D Bicycle.
In Section \ref{sec:conclusion}, we make brief conclusions.

\section{Background} \label{sec:background}

HJ reachability is a powerful tool for guaranteed-safety analysis of nonlinear system dynamics under adversarial inputs, compatible with arbitrary shapes of target sets.
Here, we present the necessary setup for HJ reachability computation, and introduce the projection operations used in our method.

\subsection{System Dynamics}

Let $z \in \R^{n}$ represent the state and $s$ represent time. The system dynamics is described by the following ODE:
\begin{align}\label{equ:nonlinear_dyns}
    &\dot{z} = \frac{\mathrm{d} z}{\mathrm{d} s} = f(z,u,d), & s \in [\sinit, 0], \sinit \leq 0 \quad
     u \in \mathcal U, d \in \mathcal D
\end{align}

The $u(\cdot)$ and $d(\cdot)$ denote the control function and disturbance function.
For any fixed $u$ and $d$, the dynamics $f: \mathbb{R}^{n} \times \mathcal{U} \times \mathcal{D} \rightarrow \mathbb{R}^{n}$ is assumed to be uniformly continuous, bounded and Lipschitz continuous with respect to all arguments; thus, a unique solution to \eqref{equ:nonlinear_dyns} exists given $u$ and $d$. 
With opposing objectives, the control and disturbance are modeled as opposing players in a differential game \cite{mitchell2005time}.

The solution for \eqref{equ:nonlinear_dyns}, or trajectory, is denoted as $\zeta(s; z, \sinit, u(\cdot), d(\cdot)) : [\sinit,0] \rightarrow \mathbb{R}^{n}$, which starts from state $z$ at time $\sinit$ under control $u$ and disturbance $d$. 
$\zeta$ satisfies \eqref{equ:nonlinear_dyns} almost everywhere with initial condition $\zeta(\sinit; z, \sinit, u(\cdot), d(\cdot)) = z$.

\subsection{Hamilton-Jacobi Reachability}

Given a target $\mathcal{T}$ to avoid, the BRT is the set of states from which there exists a disturbance such that entering the target during the time horizon of duration $\left | \sinit \right |$ is inevitable despite the best control. 
This is illustrated in Fig. \ref{fig:target_brt}.

\begin{figure}
    \vspace{5pt}
    \centering
    \includegraphics[width=0.5\linewidth]{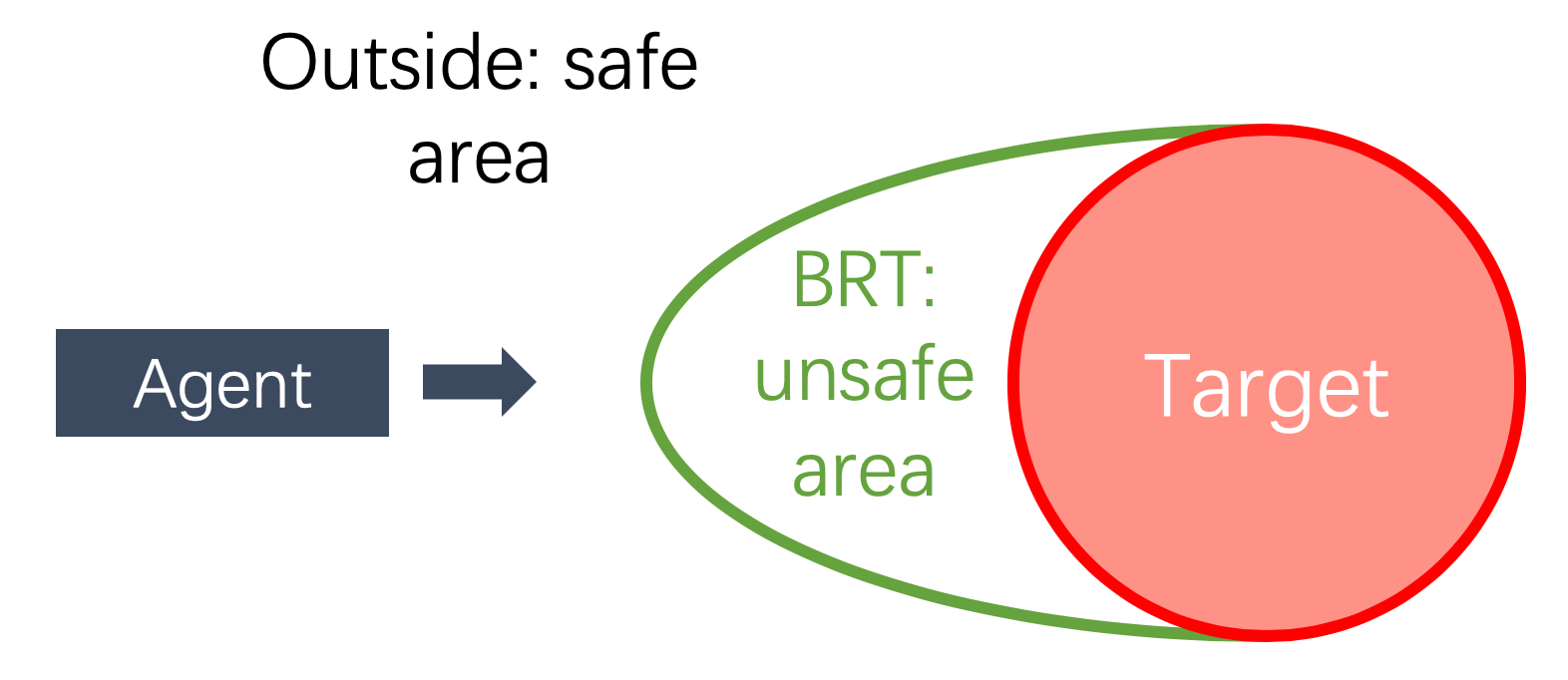}
    \caption{An example of a target and its BRT. To avoid the target, the agent should stay outside of the BRT.}
    \label{fig:target_brt}
    \vspace{-1em}
\end{figure}

The definition of the minimal BRT $\bar{\mathcal{A}}(s)$ is as follows:
\begin{align}
    \bar{\mathcal{A}}(s) = \{z: \exists d(\cdot) \in \mathbb{D}, \forall u(\cdot) \in \mathbb{U}, \exists s \in [\sinit,0],
    \nonumber \\
    \zeta(s;z,\sinit,u(\cdot)) \in \mathcal{T} \}
\end{align}

In the HJ formulation, the target set is represented as the sub-level set of some function $l(z)$, where $z\in\mathcal T \Leftrightarrow l(z) \le 0$. 
Then, the BRT in HJ reachability becomes the sub-level set of a value function $V(z,s)$ defined as below:
\begin{equation}\label{equ:final_valfunc}
    V(z,s):=\min_{d(\cdot)} \max_{u(\cdot)} \min_{s \in [s_{0}, 0]} l(\zeta(0; z,s,u(\cdot), d(\cdot)))
\end{equation}

The value function $V(z,s)$ can be obtained as the viscosity solution of the following HJ partial differential equation:
\begin{align}\label{equ:hj_equation}
    \min \{ D_{s} V(z,s)+H(z, \nabla V(z,s)), V(z,0)-V(z,s) \}=0, \nonumber \\
    V(z,0)=l(z), s \in [\sinit, 0]
\end{align}
\begin{equation}\label{equ:hamiltonian}
    H(z, \nabla V(z, s)) = \min_{d(\cdot)} \max_{u(\cdot)} \nabla V(z,s) ^\top f(z,u)
\end{equation}

The level set method \cite{mitchell2000level} is a computation tool to solve \eqref{equ:hj_equation} in the discretized state space. 
Recently, toolboxes \cite{mitchell2007toolbox, Tanabe} have been developed to numerically compute BRTs using the level set method.

\subsection{Projection}

We present two kinds of projection operations for manipulating value functions of BRTs between high dimension spaces and their low dimension subspaces. 
Let $V(x, y): \mathbb{R}^{n_{x}+n_{y}} \rightarrow \mathbb{R}$ be a value function in $n_{x}+n_{y}$ dimension space, where $x \in \mathbb{R}^{n_{x}}$ and $y \in \mathbb{R}^{n_{y}}$. 
Given a BRT represented by $V(x, y)$, we define the projected BRT in its $n_{x}$ dimension subspace by the value function $W(x) : \mathbb{R}^{n_{x}} \rightarrow \mathbb{R}$, where
\begin{equation}\label{equ:projection}
   W(x) =  \min_{y} V(x, y)
\end{equation} 

Given $W(x)$ in $n_{x}$-dimensional space, we define the value function $V(x,y)$ representing the back projected BRT as
\begin{equation}\label{equ:back_projection}
     V(x,y) = W(x), \ \forall y \in \mathbb{R}^{n_{y}}.
\end{equation} 


\section{Methodology} \label{sec:chain_model}







Reachability analysis relies on an accurate model of the robotic system under consideration, but accurate models tend to be high-dimensional. 
This often makes HJ reachability intractable due to the curse of dimensionality. 

In this section, we first present a novel method to decompose the dynamical system in \eqref{equ:nonlinear_dyns} into several coupled subsystems based on a State Dependency Graph. 
Then, we provide an algorithm for computing BRTs with these subsystems. 
Finally, we prove that our method produces conservative BRT approximations that guarantee safety, analyze its computational time and space complexity, and discuss the constraints for target sets.

\subsection{System decomposition}\label{sec:decomposition}

\subsubsection{\textbf{State Dependency Graph}}

Let $S$ be the set of states, $S=\{ z_{i} \}_{i=1}^n$.
We first define the notion of ``state dependency'': for some states $z_{i}, z_{j} \in S$, $z_{i}$ depends on $z_{j}$ in $f(z,s)$ means $\frac{\mathrm{d} z_{i}}{\mathrm{d} s}$ is a function of $z_{j}$.

In order to clarify these dependency relationships between each state of $S$ in $f(z,s)$, we define a directed State Dependency Graph $G = (S, E)$. 
The set of vertices is denoted $S$, and contains all state variables. 
If some state component $z_{i} \in S$ depends on $z_{j}$ in $f(z,s)$, then the graph $G$ would have a directed edge from $z_{i}$ to $z_{j}$, $(z_i, z_j) \in E$.

Often, high-dimensional dynamics contain chains of integrators. 
Thus, we consider a running example: 4D Quadruple Integrator, whose 
dynamics are as follows:
\begin{equation}\label{equ:4dint}
\begin{bmatrix}
\dot{z_{1}}\\ 
\dot{z_{2}}\\ 
\dot{z_{3}}\\ 
\dot{z_{4}}
\end{bmatrix}
=
\begin{bmatrix}
z_{2} + d\\ 
z_{3}\\ 
z_{4}\\ 
u
\end{bmatrix},  u \in \mathcal U, d \in \mathcal D,
\end{equation}

\noindent where $u$ and $d$ denote the control and disturbance.
For the system in \eqref{equ:4dint}, $S= \{ z_{1}, z_{2}, z_{3}, z_{4} \}$, $E=\{(z_{1}, z_{2}),(z_{2}, z_{3}),(z_{3}, z_{4})\}$, and its State Dependency Graph $G = (V, E)$ is shown in Fig. \ref{fig:sdg_4dint}.

    
\subsubsection{\textbf{Choosing coupled subsystems}}

Given the State Dependency Graph $G$ and the computational space constraint that each subsystem can be at most $p$-dimensional, we can decompose the full system $S$ into several coupled subsystems $S_1, S_2, \ldots, S_m$ whose states denoted as $x_{1}, x_{2},\ldots,x_{m}$ respectively, with the following properties:
\begin{itemize}
    \item In every subsystem, each state should depend on or be depended on by at least one other state.
    \item Every subsystem should include no more than $p$ states, where $p$ is chosen based on considerations such as computational resources available
    \item Subsystems should be ``chained'': each subsystem should share at least one state with another subsystem.
\end{itemize}

As a result, the decomposed system is represented by connected subgraphs of $G$ each representing a subsystem.
Let $\si$ be the set of state variables included in the $i^{th}$ subsystem, and let $\sic$ be the set of states that are not included in $i^{th}$ subsystem, i.e. $\forall i, \sic = S \setminus \si$.

For example, suppose that one requires the maximum dimensionality of subsystems to be two, $p=2$. 
We can decompose the 4D Quadruple Integrator into 3 subsystems $x_{1}, x_{2}, x_{3}$ with $S_{1}=\{ z_{1}, z_{2} \}, S_{2}=\{ z_{2}, z_{3} \}, S_{3}=\{ z_{3}, z_{4} \}$.
The result of the decomposition is shown in Eq. \eqref{equ:4dint_decompose}, and the corresponding State Dependency Graph representing subsystems is illustrated in Fig. \ref{fig:decompose_sdg_4dint}.
\begin{align}\label{equ:4dint_decompose}
    S_1: \dot{x_{1}}&=
    \begin{bmatrix}
        \dot{z_{1}}\\ 
        \dot{z_{2}}
    \end{bmatrix}
    =
    \begin{bmatrix}
        z_{2}+d\\ 
        z_{3}
    \end{bmatrix}, d \in \mathcal D, z_{3}(s) \in R_{z_{3}}(z_{2}, s)
\nonumber \\
    S_2: \dot{x_{2}}&=
    \begin{bmatrix}
        \dot{z_{2}}\\ 
        \dot{z_{3}}
    \end{bmatrix}
    =
    \begin{bmatrix}
        z_{3}\\ 
        z_{4}
    \end{bmatrix}, z_{4}(s) \in R_{z_{4}}(z_{3}, s)
\nonumber \\
    S_3: \dot{x_{3}}&=
    \begin{bmatrix}
        \dot{z_{3}}\\ 
        \dot{z_{4}}
    \end{bmatrix}
    =
    \begin{bmatrix}
        z_{4}\\ 
        u
    \end{bmatrix}, u \in \mathcal U
\end{align}

There may be missing state components in subsystems, e.g. $z_3 \notin S_1$. 
To guarantee safety, we assume the worst case for the missing states by treating them as virtual disturbances, which leads to an over-approximated BRT that is conservative in the right direction \cite{mitchell2003overapproximating}.
To avoid excessive conservatism, the virtual disturbances are bounded by concurentlly computed BRT over-approximations from other subsystems.

Formally, consider some subsystem $S_i$, and let $R_{z_j}(\xii,s)$ denote the range of the missing state $z_j \notin S_i$.
Suppose $z_j \in S_k$ with $S_k$ being chained with $S_i$, $S_k \cap S_i \neq \emptyset$.
Note that $k$ is not unique, as $z_j$ may be a state of many different subsystems.
Furthermore, let $W_{k}(x_k,s)$ be the value function for the subsystem $S_k$ at the time $s$.
  Then, $R_{z_j}(\xii,s)$ is determined from $W_{k}(x_k,s)$ as follows:
\begin{multline}\label{equ:find_missing_states}
    R_{z_j}(\xii,s) = \left \{ z_j | W_{k}(x_{k}, s ) \leq 0,\right. 
    \\
    \left. \forall k \text{ such that } z_j \in S_{k} \wedge S_{k} \cap S_{i} \neq \emptyset \right \}
\end{multline}

Our method also provides a simple way to adjust the trade off between computational burden and degree of conservatism. 
Depending on different requirements for computational resources and approximation accuracy, one can easily switch between having higher-dimensional subsystems (larger $p$) for which BRTs are more accurate, and having lower-dimension subsystems (smaller $p$) for which BRTs are faster to compute.




\begin{table*}[!htb]
\vspace{5pt}
\caption{Decomposition suggestions for 5D Car and 6D Planar Quadrotor}
\resizebox{\linewidth}{!}{
\begin{tabular}{|c|c|c|c|c|}
\hline
\label{table1}
System configuration                                                                                                                                                           & System dynamics & State Dependency Graph & Decomposed State Dependency Graph & Time and space \\ \hline
\begin{tabular}[c]{@{}l@{}}
5D Car\\ $(x,y)$-position\\ $\theta$ - heading\\ $v$ - speed\\ $\omega$ - turn rate\\ $u_{a}$ - accel. control \\ $u_{\alpha}$ - ang. accel. control \\
\end{tabular}                         &\begin{math}\begin{bmatrix}
\dot{x}\\ 
\dot{y}\\ 
\dot{\theta}\\ 
\dot{v}\\
\dot{\omega}
\end{bmatrix}
=
\begin{bmatrix}
v \cos \theta\\ 
v \sin \theta\\ 
\omega\\
u_{a}\\
u_{\alpha}
\end{bmatrix} \end{math}                &   
    \begin{minipage}{0.18\textwidth}
    \centering
    \includegraphics[scale=0.35]{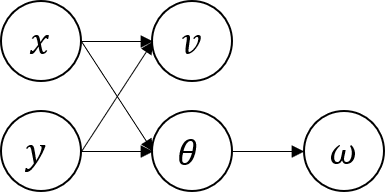}
\end{minipage}           & 
  \begin{minipage}{0.35\textwidth}\centering
  \vspace{0.2cm}
    \includegraphics[scale=0.35]{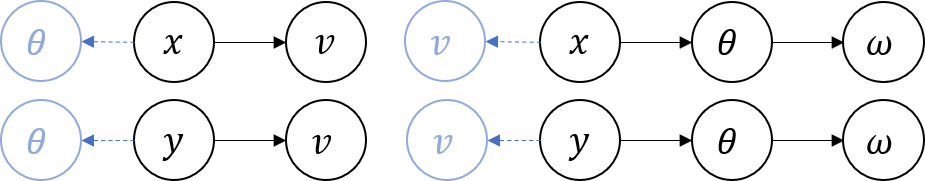}
    \vspace{0.2cm}
\end{minipage}                                      &  \begin{tabular}[c]{@{}l@{}}
 Ground truth: \\ 
 both $O(k^{5})$ \\
 \\
Decomposition: \\
$O(k^{4})$ and $O(k^{3})$  \\
\end{tabular}                                           \\ \hline
\begin{tabular}[c]{@{}l@{}}6D Planar Quadrotor\\ $(x,y)$-position\\ $(v_{x},v_{z})$ - velocity\\ $\theta$ - pitch\\ $\omega$ - pitch rate\\ $u_{T}$ - thrust control\\ $u_{\tau}$ - ang.accel.control \end{tabular} & \begin{math}\begin{bmatrix}
\dot{x}\\ 
\dot{z}\\ 
\dot{v_{x}}\\ 
\dot{v_{z}}\\
\dot{\theta}\\
\dot{\omega}
\end{bmatrix}
=
\begin{bmatrix}
v_{x}\\ 
v_{z}\\ 
-u_{T}\sin \theta\\ 
u_{T} \cos \theta - g\\
\omega\\
u_{\tau}
\end{bmatrix} \end{math}                             &    \begin{minipage}{0.18\textwidth} \centering
    \includegraphics[scale=0.35]{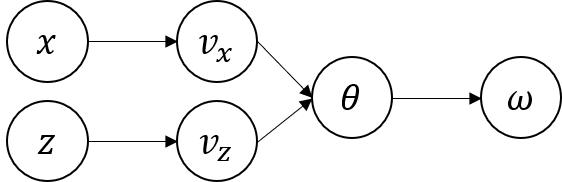}
\end{minipage}            &
\begin{minipage}{0.35\textwidth}\centering
    \vspace{0.2cm}
    \includegraphics[scale=0.35]{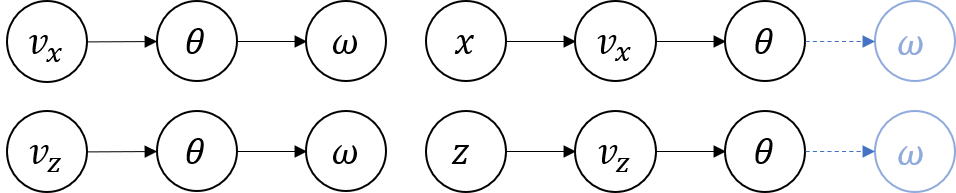}
    \vspace{0.2cm}
\end{minipage}            &  \begin{tabular}[c]{@{}l@{}}
 Ground truth: \\
 both $O(k^{6})$ \\
 \\
Decomposition: \\ 
$O(k^{4})$ and $O(k^{3})$  \\
\end{tabular}                                \\ \hline
\end{tabular}
}
\end{table*}

\begin{algorithm*}
  \caption{Approximating full-dimensional BRTs with chained subsystems}
    \begin{algorithmic}[1]
        \Require System dynamics $f(z,u,d)$ described as \eqref{equ:nonlinear_dyns} and a function $l(z)$ representing the target set $\mathcal T$\\
        Initialize the full-dimensional final time value function $V(z, 0)$ as \eqref{equ:hj_equation} \\
        Decompose the entire system into chained subsystems $S_{1}, S_{2}, S_{3},...S_{n}$, based on Section \ref{sec:decomposition} \\
        Initialize the final time value functions $\wifinal$ for each subsystem $\si$ based on \eqref{equ:subsys_final_valfunc} 
        \For{$(s = 0; s \geq s_{0}; s=s- \Delta s)$}
            \State \textbf{for each subsystem $S_{i}$} 
            \Indent
                \State Find the range $R_{z_{j}}(\xii,s)$ of the missing states $z_{j}$ at time $s$ based on \eqref{equ:find_missing_states}
                \State Obtain $\wis$ by solving the HJ equation in \eqref{equ:hj_subsys_each_time}
            \EndIndent
        \EndFor \\
        Obtain the approximated $\tilde V(z,s_{0})$ based on \eqref{equ:init_time_valfunc} \\
        Obtain the approximated full-dimensional BRT from the zero sub-level set of $\tilde V(z,s_{0})$ \\
        For any time $s$, obtain the optimal controller as \eqref{equ:compute_full_val_at_s} and \eqref{eq:opt_ctrl_approx}
\end{algorithmic}
\label{Algo1}
\end{algorithm*}
Although in general it may not be possible to decompose arbitrary dynamical systems in the form of \eqref{equ:nonlinear_dyns} in a way that saves computation time, our approach is very flexible and can often successfully decompose many realistic system dynamics. 
We demonstrate the method by decomposing the high-dimensional, tightly coupled 6D Bicycle in Section \ref{subsection:bicycle}. 
We also present decomposition suggestions for two other common system dynamics, 5D car \cite{rajamani2011vehicle} and 6D planar quadrotor \cite{singh2018robust}, in TABLE \ref{table1}.

\subsection{Backward Reachable Tube Computation}

We now present the procedure for over-approximating BRTs with low-dimensional chained subsystems $S_1,\ldots,S_m$. 
Given the target set $\mathcal{T}$ and the corresponding final condition to the HJ PDE \eqref{equ:hj_equation}, $l(z) = V(z,0)$, we  project the full-dimensional BRT onto the subspace of each subsystem $S_{i}$, and initialize the final time value function $\wifinal$ for the subsystem $\si$ using the projection operation in \eqref{equ:projection} as follows:
\begin{equation}\label{equ:subsys_final_valfunc}
  \wifinal =  \min_{z_{i} \in \sic} V(z,0)
\end{equation}

Then, given $W_{i}(x_{i},t)$ for some $t$, we compute the value function $\wis$ backwards in time for each subsystem following standard HJ PDE theory and level-set methods, while treating missing variables as virtual disturbances with appropriate bounds.
For each time step $s \in [t-\Delta s, t]$, $W_{i}(x_{i},s)$ is the viscosity solution of the following HJ partial differential equation:
\begin{equation}\label{equ:hj_subsys_each_time}
\begin{split}
    \min \{D_{s} \wis +H(x_{i}, \nabla \wis), \\ \wifinal - \wis \} =0,
\end{split}
\end{equation}

The Hamiltonian is given by
\begin{multline}\label{equ:subsys_hamiltonian}
    H(x_{i}, \nabla \wis) = \\
    \min_{\substack{d\in \mathcal D \\ z_{k} \in R_{z_{k}}(\xii,s), \forall z_{k} \in \sic}} \max_{u\in\mathcal U} \sum_{ z_{j} \in \si} \frac{\partial \wis}{\partial z_{j}} \cdot \frac{\partial z_{j}}{\partial s}
\end{multline}

\noindent where $R_{z_{i}}(\xii,s)$ is the range of missing states $\{z_{k}\}$ given in \eqref{equ:find_missing_states}.



This procedure starts at $t = -\Delta s$, and finishes when $W_{i}(x_{i}, s_0)$ is obtained.
Finally, we take the maximum of all the $\wiinit$ as the over-approximation of the full-dimensional initial time $\tilde V(z,s_{0})$:
\begin{equation}\label{equ:init_time_valfunc}
    \tilde V(z,s_{0}) = \max_{i} \ \wiinit
\end{equation}

In general for any time $s$, we also have over-approximated
\begin{equation}\label{equ:compute_full_val_at_s}
    \tilde V(z,s) = \max_{i} \ \wis.
\end{equation}



The optimal controller is given by
\begin{equation} \label{eq:opt_ctrl_approx}
    u^{*}(s) = \underset{u}{\argmax} \ \nabla \tilde V(z,s) ^\top f(z,u).
\end{equation}



The computation process is summarized in \textbf{Algorithm \ref{Algo1}}.

Consider our running example, the 4D Quadruple Integrator in \eqref{equ:4dint} and its subsystems in \eqref{equ:4dint_decompose}.
In particular, consider the BRT computation for subsystem $S_1 = \{z_1, z_2\}$. 
At the time $s$, the HJ equation in \eqref{equ:hj_subsys_each_time} for subsystem $S_{1}$ becomes
\begin{equation}\label{equ:hj_4dint_subsys1}
\begin{split}
    \min \left \{ \frac{\partial W_{1}(x_{1},s)}{\partial s} +
    \min_{z_{3} \in R_{z_{3}}(z_{2}, s)}  \  (\frac{\partial W_{1}(x_{1},s)}{\partial z_{1}} z_{2} \right.
    \\
    \left. + \frac{\partial W_{1}(x_{1},s)}{\partial z_{2}} z_{3}  ), W_{1}(x_{1},0)-W_{1}(x_{1},s) \right \}  = 0
\end{split}
\end{equation}

Here for the subsystem $S_{1}$ in \eqref{equ:4dint_decompose}, given $z_{2}$, we are able to find the range of $z_{3}$ in the subsystem $S_{2}$. Let $W_{2}(x_{2},s)$ be the value function for the subsystem $S_{2}$ at the time $s$, the range of $z_{3}$ given $z_{2}$ will be defined as $R_{z_{3}}(z_{2}, s)$: 
\begin{equation} \label{equ:4dint_search_dist}
    R_{z_{3}}(z_{2}, s) := \left \{ z_{3} | W_{2}(x_{2}, s) \leq 0 \right \}
\end{equation}

A graphical interpretation of \eqref{equ:4dint_search_dist} is in Fig. \ref{fig:search_dist}. 
For a specific grid point of $(z_{1}, z_{2})=(a, b) \in \mathbb{R}^{2}$ in the subsystem $S_{1}$, the range of the missing state $R_{z_{3}}(z_{2}, s)$ can be drawn from the subsystem $S_{2}$ with the corresponding $z_{2} = b$.

\begin{figure}
    \centering
    \includegraphics[width=0.6\linewidth]{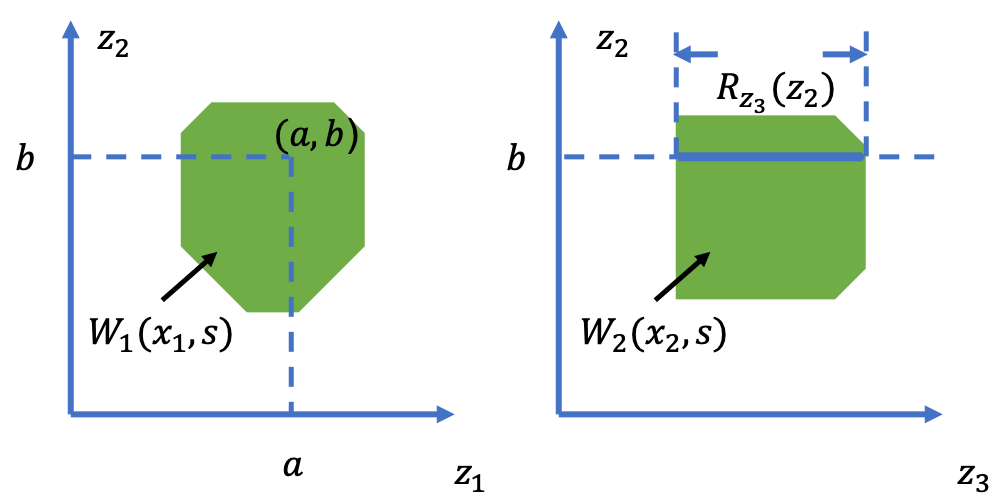}
    \caption{Searching missing states for 4D Quadruple Integrator. Left: a BRT described by $W_{1}(x_{1},s)$ for subsystem $S_{1}$. Right: a BRT described by $W_{2}(x_{2},s)$ for subsystem $S_{2}$. When solving on the grid point of $(z_{1}, z_{2})=(a, b)$ in subsystem $S_{1}$ at time $s$, we find the range of missing state  $R_{z_{3}}(z_{2}, s)$ inside the BRT from subsystem $S_{2}$.}
    \label{fig:search_dist}
    \vspace{-1.5em}
\end{figure}

\subsection{Proof and Discussions}

\subsubsection{\textbf{Proof of Correctness}}

In this section, we show that the BRT generated from our method is an over-approximation of the true BRT obtained from \eqref{equ:hj_equation}.
Let $\tilde V_{i}(z, s)$ denote the full-dimensional value function that back projected from $\wis$ at the time $s$, based on the projection operation in \eqref{equ:back_projection}:
\begin{equation}\label{equ:back_project_initialize}
    \tilde V_{i}(z, s) = \wis, \forall z_{i} \in \sic.
\end{equation}

Because at any time $s$, we maximize over $\wis$ to obtain the over-approximation, to prove the following Theorem 1 is sufficient to prove that each approximate value function $\tilde V_i$ is no larger than the true value function $V$.

\begin{theorem}
For any subsystem $S_{i}$ at any time step $s \in [t-\Delta s, t]$, $\tilde V_{i}(z, s) \leq V(z,s)$
\end{theorem}

\begin{proof}
We prove this by mathematical induction. 
For any subsystem $S_{i}$, we first show that the Theorem at final time $s=0$ is true. 
Then we prove that for any time step $s \in [t-\Delta s, t]$, if $\tilde V_{i}(z, t) \leq V(z,t)$, we will have $\tilde V_{i}(z, t - \Delta s) \leq V(z,t - \Delta s)$.

At the final time $s=0$, $\wifinal$ is initialized as \eqref{equ:subsys_final_valfunc} and $\tilde V_{i}(z, s)$ is initialized as \eqref{equ:back_project_initialize}, thus trivially we have 
\begin{equation}\label{equ:final_inequality}
    \tilde V_{i}(z, 0) \leq V(z, 0).
\end{equation}

For any time step $s \in [t-\Delta s, t]$, $V(z,s)$ is the viscosity solution of \eqref{equ:hj_equation} with final value $V(z, t)$. 
Let $\tilde V_{i}^{*}(z, t-\Delta s)$ be the viscosity solution of \eqref{equ:hj_equation} with final value $\tilde V_{i}(z, t)$.
Since 
\begin{equation*}
    \tilde V_{i}(z, t) \leq V(z,t),
\end{equation*}{}
we have
\begin{equation}\label{equ:inequality1}
\vspace{5pt}
    \tilde V_{i}^{*}(z, t-\Delta s) \leq V(z, t-\Delta s)
\end{equation}

Let $W_{i}(\xii, s)$ be the viscosity solution of \eqref{equ:hj_subsys_each_time} at $s \in [t-\Delta s, t]$ with final value $W_{i}(\xii, t)$. 
When solving $W_{i}(\xii, t - \Delta s)$, the Hamiltonian $H(x_{i}, \nabla \wis)$ is computed as \eqref{equ:subsys_hamiltonian}. 
For comparison, when solving $\tilde V_{i}^{*}(z, t-\Delta s)$, the Hamiltonian $H(z, \nabla \tilde V_{i}^{*}(z, s))$ is computed as \eqref{equ:hamiltonian}.

Because $\tilde V_{i}(z, t)$ is back projected from $W_{i}(\xii, t)$ as \eqref{equ:back_project_initialize}, in $H(z, \nabla \tilde V_{i}(z, s))$ we have $\frac{\partial \tilde V_{i}(z, t)}{\partial z_{i}}=0, \forall z_{i} \in \sic$.
In addition, missing states are treated as disturbances in $H(x_{i}, \nabla \wis)$, so $H(x_{i}, \nabla \wis) = \min_{\forall z_{i} \in \sic} H(z, \nabla \tilde V_{i}^{*}(z, s))$. Therefore,
\begin{equation}\label{equ:inequality2}
    H(x_{i}, \nabla \wis) \leq H(z, \nabla \tilde V_{i}^{*}(z, s)), \forall z_{i} \in \sic.
\end{equation}

Thus we obtain
\begin{equation}\label{equ:inequality3}
    W_{i}(\xii, t - \Delta s) \leq \tilde V_{i}^{*}(z, t-\Delta s), \forall z_{i} \in \sic.
\end{equation}

Because $\tilde V_{i}(z, t-\Delta s)$ is back projected from $W_{i}(\xii, t - \Delta s)$ as \eqref{equ:back_project_initialize}, we have
\begin{equation}\label{equ:inequality4}
    \tilde V_{i}(z, t-\Delta s) \leq \tilde V_{i}^{*}(z, t-\Delta s).
\end{equation}

Finally, for any time step $s \in [t, t-\Delta s]$, we combine \eqref{equ:inequality1} and \eqref{equ:inequality4} and obtain
\begin{equation}
    \tilde V_{i}(z, t-\Delta s) \leq V(z, t-\Delta s).
\end{equation}



\end{proof}

\begin{figure}[ht]
\vspace{-1.5em}
\centering
\includegraphics[width=0.4\columnwidth]{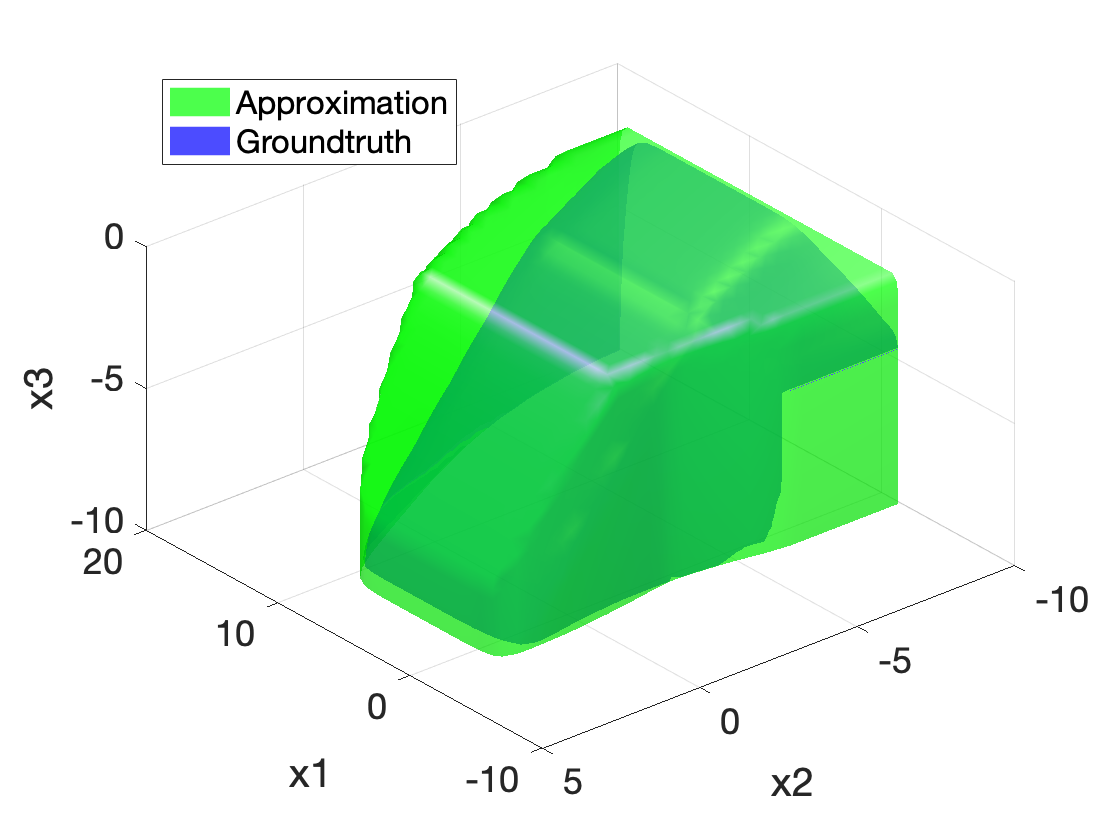}
\includegraphics[width=0.4\columnwidth]{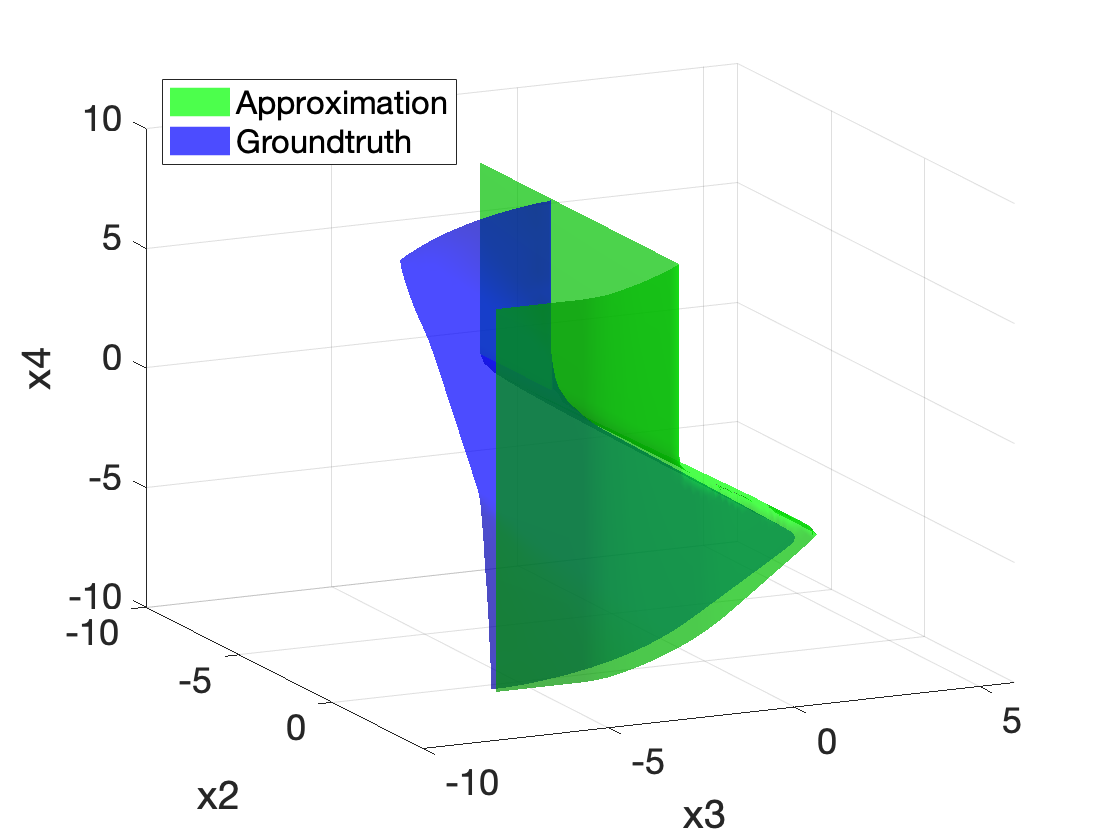}
\includegraphics[width=0.4\columnwidth]{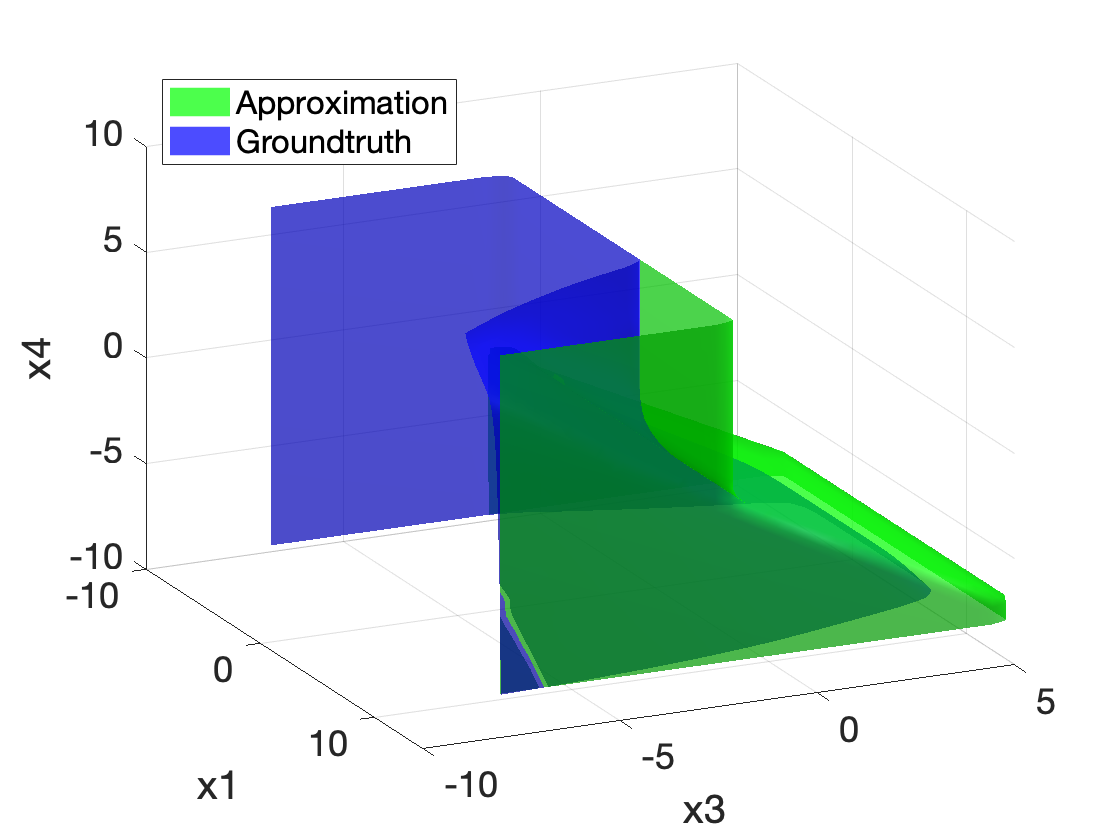}
\caption{Comparison of ground truth BRTs (blue) and our approximated BRTs (green) for 4D Quadruple Integrator at $s=-1$. From top left, top right to bottom are 3D slices at $z_{4}=-2, z_{1}=2.6, z_{2}=-4.2$.}
\label{fig:4dint_brt}
\vspace{-1em}
\end{figure}

\subsubsection{\textbf{Computation complexity}} \label{subsubsec:comp_complexity}

Let $k$ be the number of grid points in each dimension for the numerical computation.
The computational space complexity is determined by the largest dimension of subsystems. 
If each subsystem $S_i$ has $N_i$ states, the space complexity is $O(k^{\max_i N_{i}})$.

For computation time, there are two non-trivial parts: solving the HJ PDE and searching the missing states.
If the HJ PDE is solved on a grid with $O(k^{N_{i}})$ grid points in subsystem $S_i$ with a search over an $M_i$-dimensional grid, then these nested loops have a time complexity of $O(k^{N_{i}+M_i})$. 
Overall, the upper limit of the computation time will be the longest time among all subsystems, $O(k^{\max_i \{N_{i}+M_i\}})$.

\subsubsection{\textbf{Target sets}}

In our decomposition method, we require a clear boundary of the target in each subsystem.
In addition, due to shared controls and disturbances in subsystems, the entire target should be the intersections of all targets from each subsystem to ensure the conservative approximation of BRT, according to \cite{chen2018decomposition}.

\section{Numerical Experiments} \label{sec:exp}


We first demonstrate our method on the running example, 4D Quadruple Integrator.  
By comparing our approximate BRT to ground truth BRT obtained from the full-dimensional computation, we show that our method maintains the safety guarantee without introducing much conservatism. 
Then we focus on the higher-dimensional, heavily coupled 6D Bicycle model \cite{singh2018robust}, and present for the first time a conservative but still practically useful BRT in a realistic simulated autonomous driving scenario.
All numerical experiments are implemented on an AMD Ryzen 9 3900X 12-Core Processor with ToolboxLS \cite{mitchell2007toolbox} and helperOC toolbox.

\subsection{4D Quadruple Integrator}


The system dynamics of the 4D Quadruple Integrator and its decomposed subsystems are given in \eqref{equ:4dint} and \eqref{equ:4dint_decompose}. 
Starting from the target set $\mathcal{T}$ in \eqref{eq:4dtarget}, we compute the approximate BRT for a time horizon of $1.0$ second,
\begin{equation} \label{eq:4dtarget}
    \mathcal{T} := \{(z_{1}, z_{2}, z_{3}, z_{4}) \ | \ -6 < z_{1} < 6, z_{2} < -4, z_{3} < -2 \}
\end{equation}

In Fig. \ref{fig:4dint_brt}, we visualize the 4D BRT through 3D slices at the initial time $s=-1.0$. 
From top left, top right to bottom, our approximated BRTs (green) and the ground truth BRTs (blue) are shown, at the slices of $z_{4} = -2$, $z_{1} = 2.6$, and $z_{2} = -4.2$ respectively. 
The results show that our approximated BRTs are similar in shape to the ground truth BRT while being a little bigger, which indicates that our results are conservative in the right direction: if a state is outside of the approximate BRT, it is guaranteed to be safe.

According to Sec. \ref{subsubsec:comp_complexity}, for the 4D Quadruple Integrator, we need to solve on a two dimensional grid with a search on another two dimensional grid in subsystem $S_{1}$ and $S_{2}$, thus the computation space and time are $O(k^{2})$ and $O(k^{3})$.
To compare, computing ground truth BRTs in the full dimension space will cost $O(k^{4})$ both on space and time.

In our experiment, it takes $2.5$ seconds to compute approximation from decomposition, while it takes $420$ seconds to compute the ground truth in full dimension.

\subsection{6D Bicycle}\label{subsection:bicycle}
    

To illustrate the utility of our method on decomposing high-dimensional and heavily coupled systems, we present the first practically usable minimal BRT computation for the 6D Bicycle, a model widely used to approximate the behaviour of four-wheeled vehicles such as autonomous cars.

\subsubsection{\textbf{Problem Setup}}

The system dynamics is given in \eqref{equ:6d_bidycle}. 
$X$ and $Y$ denote position in the global frame,
$\psi$ denotes the orientation angle with respect to the $X$ axis\footnote{Computation bound for $\psi$ is $[\pi/4, 9\pi/4]$}, $v_{x}$ and $v_{y}$ denote the longitudinal and lateral velocities, and $\omega$ denotes the angular speed. 
The controls are $\delta_{f}$ and $a_{x}$, which represent the steering angle and longitudinal acceleration, respectively.
\begin{equation} \label{equ:6d_bidycle}
    \begin{bmatrix}
\dot{X}\\ 
\dot{Y}\\ 
\dot{\psi}\\ 
\dot{v_{x}}\\
\dot{v_{y}}\\
\dot{\omega}
    \end{bmatrix}
=
    \begin{bmatrix}
v_{x} \cos{\psi} - v_{y} \sin{\psi}\\ 
v_{x} \sin{\psi} + v_{y} \cos{\psi}\\ 
\omega\\ 
\omega v_{y} + a_{x}\\ 
-\omega v_{x} + \frac{2}{m} (F_{c,f} \cos{\delta_{f}} + F_{c, r})\\ 
\frac{2}{I_{z}}(l_{f}F_{c,f} - l_{r} F_{c,r})
    \end{bmatrix}
\end{equation}

To decompose 6D Bicycle, we set the space and time limits to be $O(k^4)$ for best possible accuracy.
Based on the State Dependency Graph for 6D Bicycle in Fig. \ref{fig:sdg_6dbicycle}, we choose the subsystems in \eqref{eq:6d_decomp} with the corresponding decomposed State Dependency Graph shown in Fig. \ref{fig:sdg_6dbicycle_decompose}, which requires $O(k^3)$ space and $O(k^4)$ time complexity.
\begin{equation}\label{eq:6d_decomp}
    \begin{split}
        x_{1} &= (X, v_{x}, v_{y}), x_{2} = (Y, v_{x}, v_{y}), x_{3} = (X, \psi), \\
        x_{4} &= (Y,\psi), x_{5} = (v_{x}, v_{y}, \omega), x_{6} = (\psi, \omega)
    \end{split}
\end{equation}
\begin{figure}[htbp]
\vspace{-1em}
\begin{subfigure}[b]{.3\linewidth}
    \parbox[][3cm][c]{\linewidth}{
    \centering
    \includegraphics[scale=0.32]{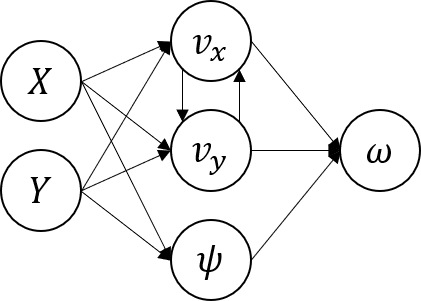}}
    \caption{}
    \label{fig:sdg_6dbicycle}
\end{subfigure}
\quad
%
%
\begin{subfigure}[b]{.6\linewidth}
    \centering
    \includegraphics[scale=0.32]{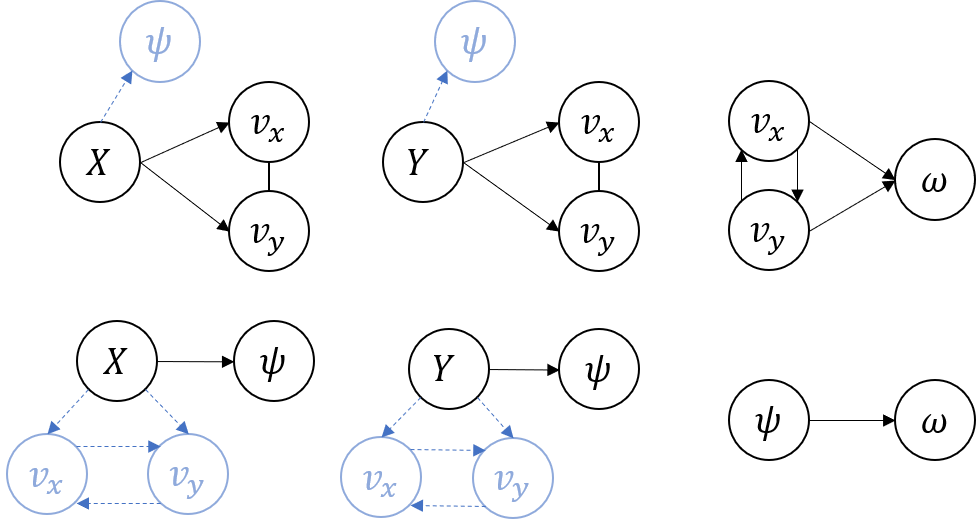}
    \caption{}
    \label{fig:sdg_6dbicycle_decompose}
\end{subfigure}
\caption{($a$) The State Dependency Graph for 6D Bicycle. ($b$) A decomposed State Dependency Graph for 6D Bicycle. The light blue vertices and edges indicate the missing states and their dependency.}
\end{figure}

We design the target set $\mathcal{T}$ with respect to $X$, $Y$, $\psi$ and $v_{x}$ in \eqref{equ:target_6d_bicycle}. 
This target set represents a one way road surrounded by an open area in a parking lot depicted as Fig. \ref{fig:target_6d_bicycle}.
In the one way road, only a positive forward speed and a forward orientation range is allowed\footnote{Combined with the computation bound, the safe orientation range are $[-\pi /4, \pi /4]$.}.
\begin{equation}\label{equ:target_6d_bicycle}
\begin{split}
     \mathcal{T} := \{(X, Y, \psi, v_{x}, v_{y}, \omega) \ | \ -6 < X < 6, \\ -2 < Y < 2, \psi < 7\pi / 4 , v_{x} < 0 \}
\end{split}
\end{equation}
\begin{figure}
    \vspace{5pt}
    \centering
    \includegraphics[width=0.65\linewidth]{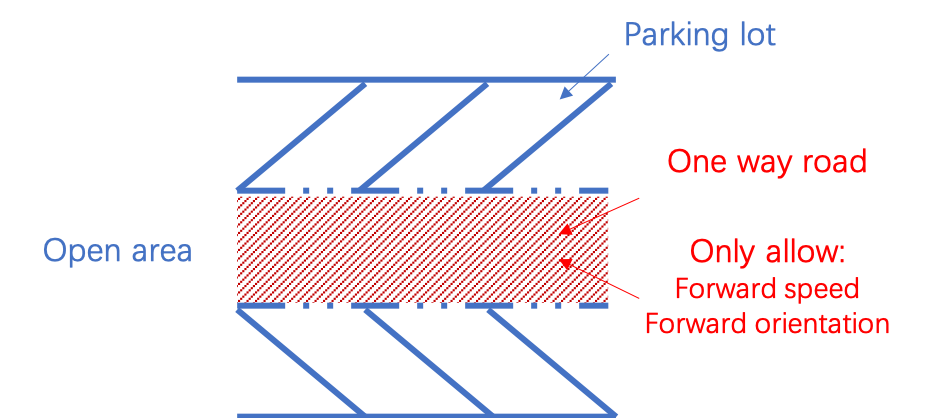}
    \caption{A target set example for 6D Bicycle. Inside a parking lot, there is a one-way road surrounded by open areas, where only a positive forward speed and a forward orientation range are allowed.}
    \label{fig:target_6d_bicycle}
    \vspace{-0.5em}
\end{figure}
We compute the BRT for a time horizon of $2$ seconds.

\subsubsection{\textbf{BRT Result}}

To visualize the 6D BRT at $t=-2.0$, we present two 3D slices in Fig. \ref{fig:6d_bicycle_brt_1}.

Fig. \ref{fig:6d_bicycle_brt_1} shows the BRT at the slice of $\psi=\pi/4,\omega = -1.1$, $v_{y}=0$ (left) and $18$ (right), indicating the range of $v_{x}$ to avoid for different $X$ and $Y$. 
As shown, the farther from the area $\{(x,y)|-6<x<6, -2<y<2)\}$, the smaller set of $v_{x}$ needs to be avoided to maintain safety. 
This is because if the agent is far from the unsafe positions, it has more time and space to slow down and adjust $v_{x}$. 
In comparison, in the right plot, the agent has a larger  $v_{y}=18$, and thus has a larger BRT in $v_{x}$ to avoid, especially in the $Y$ direction.


In our experiment, it takes $17$ minutes to compute the approximated BRT with the decomposition method.

{\color{red}
}
\begin{figure}
    \centering
    \begin{subfigure}[b]{0.4\linewidth}
    \centering
    \includegraphics[width=\columnwidth]{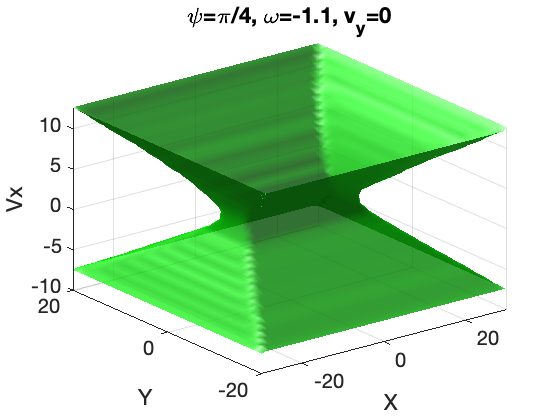}
    \label{fig:exp1_1}
    \end{subfigure}
    \begin{subfigure}[b]{0.4\linewidth}
    \centering
     \includegraphics[width=\columnwidth]{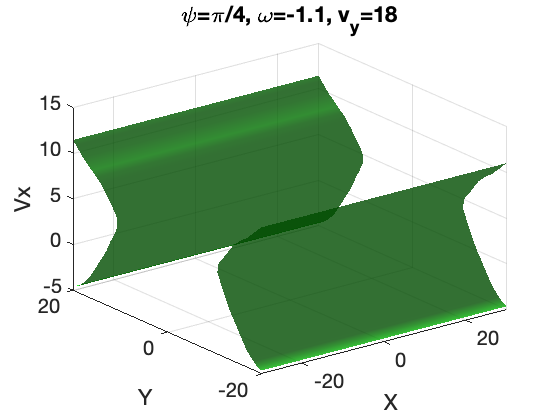}
    \label{fig:exp2_1}
    \end{subfigure}
    \caption{3D slices of $(X, Y, v_{x})$ from 6D BRT at $s=-2$. Left: slice at $\psi=\pi/4,\omega = -1.1, v_{y}=0$. Right: slice at $\psi=\pi/4,\omega = -1.1, v_{y}=18$}
    \label{fig:6d_bicycle_brt_1}

    \vspace{-1.5em}
\end{figure}

\subsubsection{\textbf{Safety-Preserving Trajectories}}

With the evolution of the BRT over time, we illustrate that trajectories synthesized using Eq. \eqref{eq:opt_ctrl_approx} are guaranteed safe when starting outside the approximated BRTs, and may enter the targets when starting inside the approximated BRTs.

{\color{red}
}
\begin{figure}[!htb]
  \centering
  \begin{subfigure}{.4\linewidth}
    \centering
    \includegraphics[width = \linewidth]{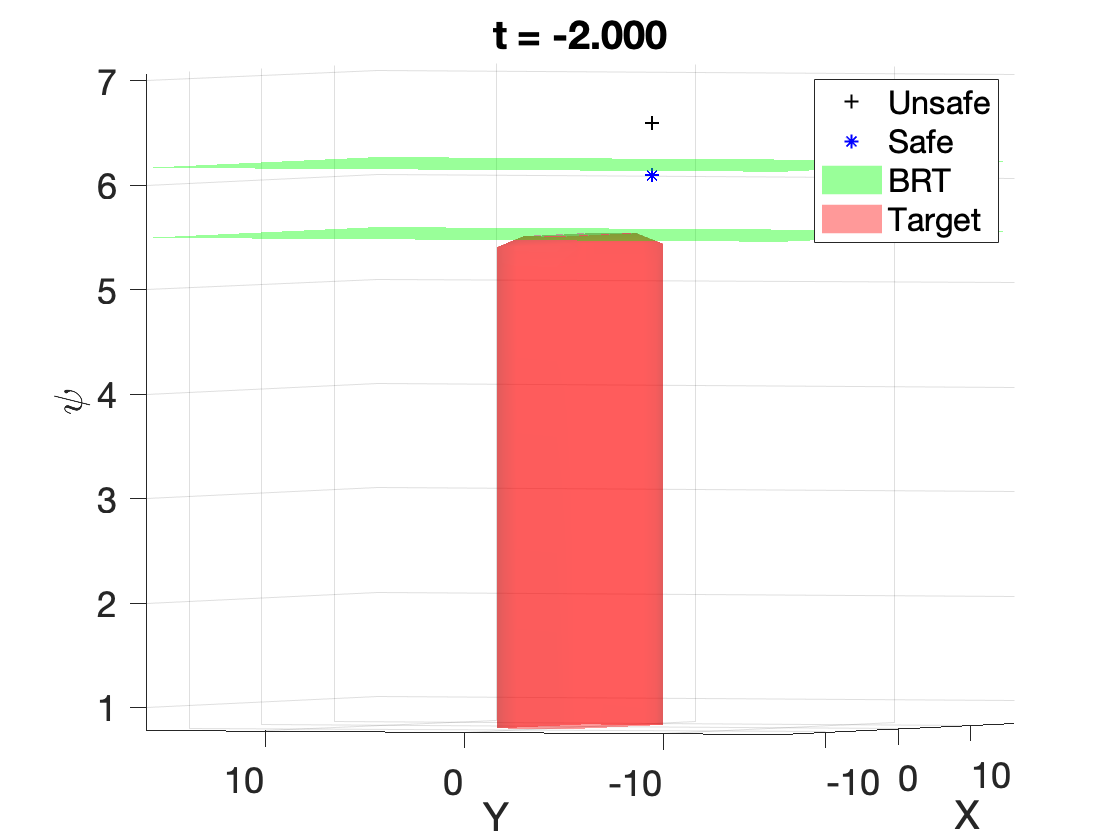}
  \end{subfigure}%
  \hspace{0.1pt}
  \begin{subfigure}{.4\linewidth}
    \centering
    \includegraphics[width = \linewidth]{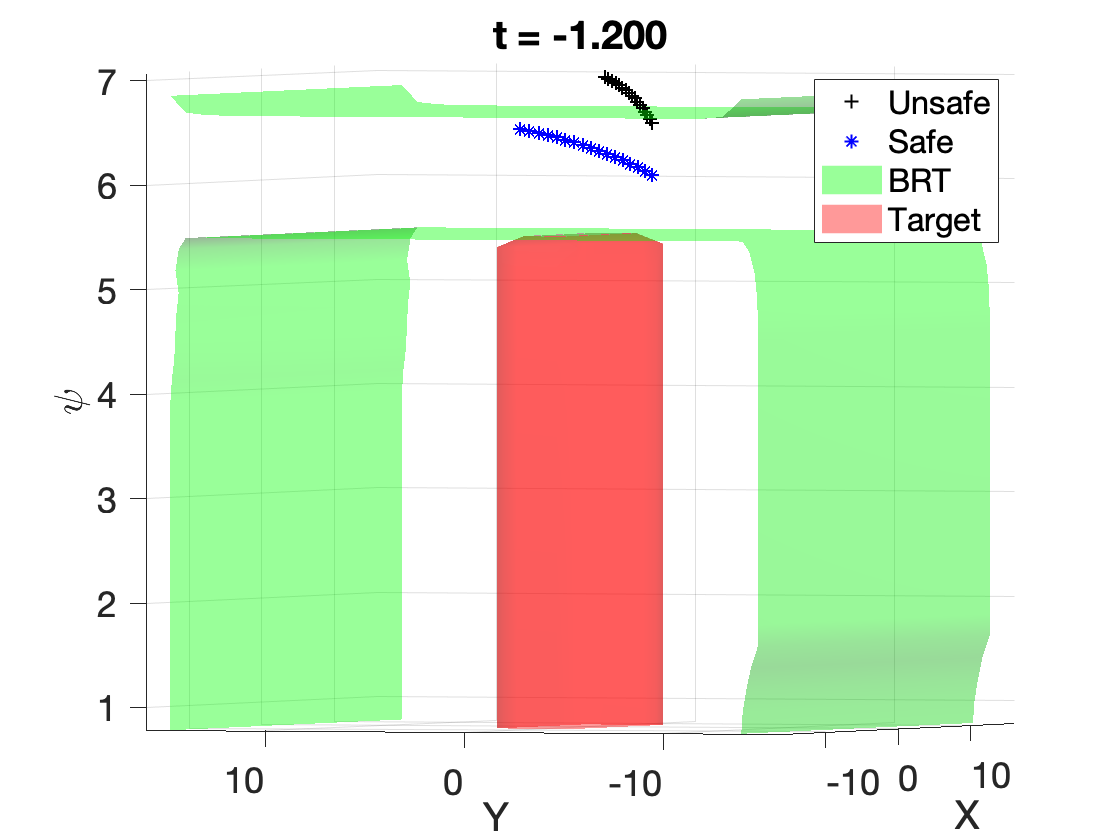}
  \end{subfigure}%
  \vfill
  \begin{subfigure}{.4\linewidth}
    \centering
    \includegraphics[width = \linewidth]{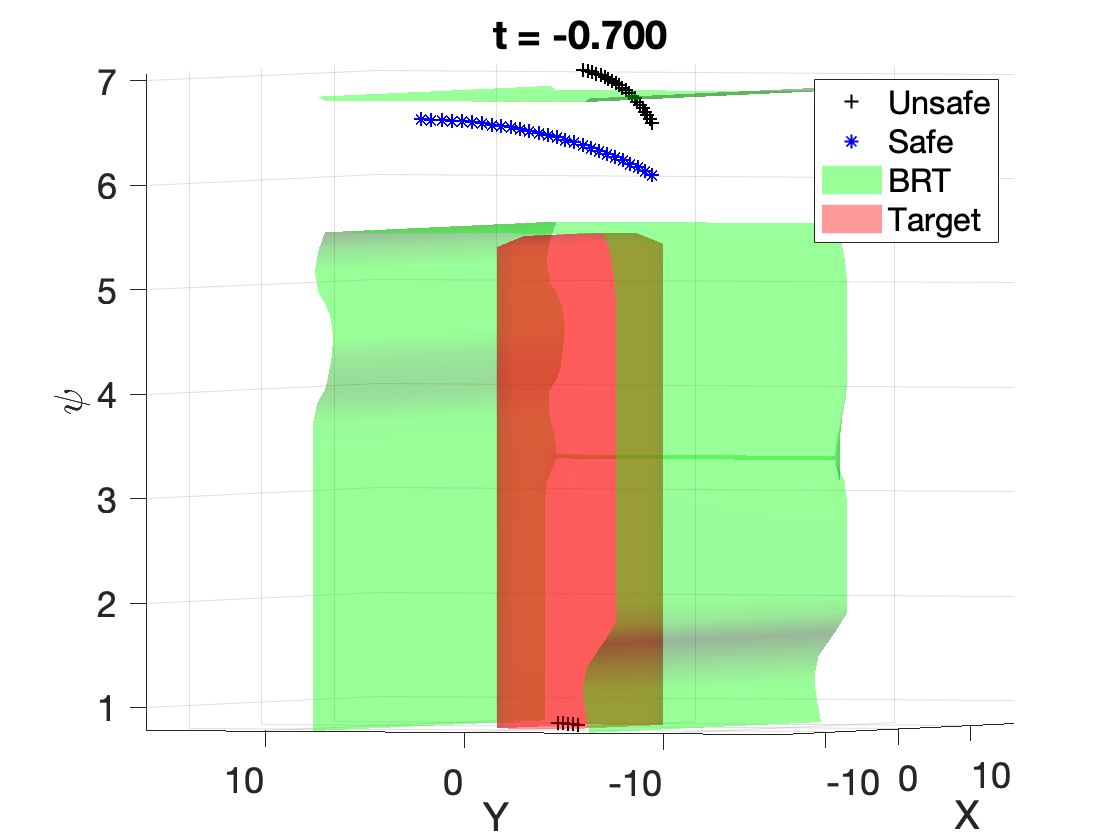}
  \end{subfigure}
  \hspace{0.1pt}
  \begin{subfigure}{.4\linewidth}
    \centering
    \includegraphics[width = \linewidth]{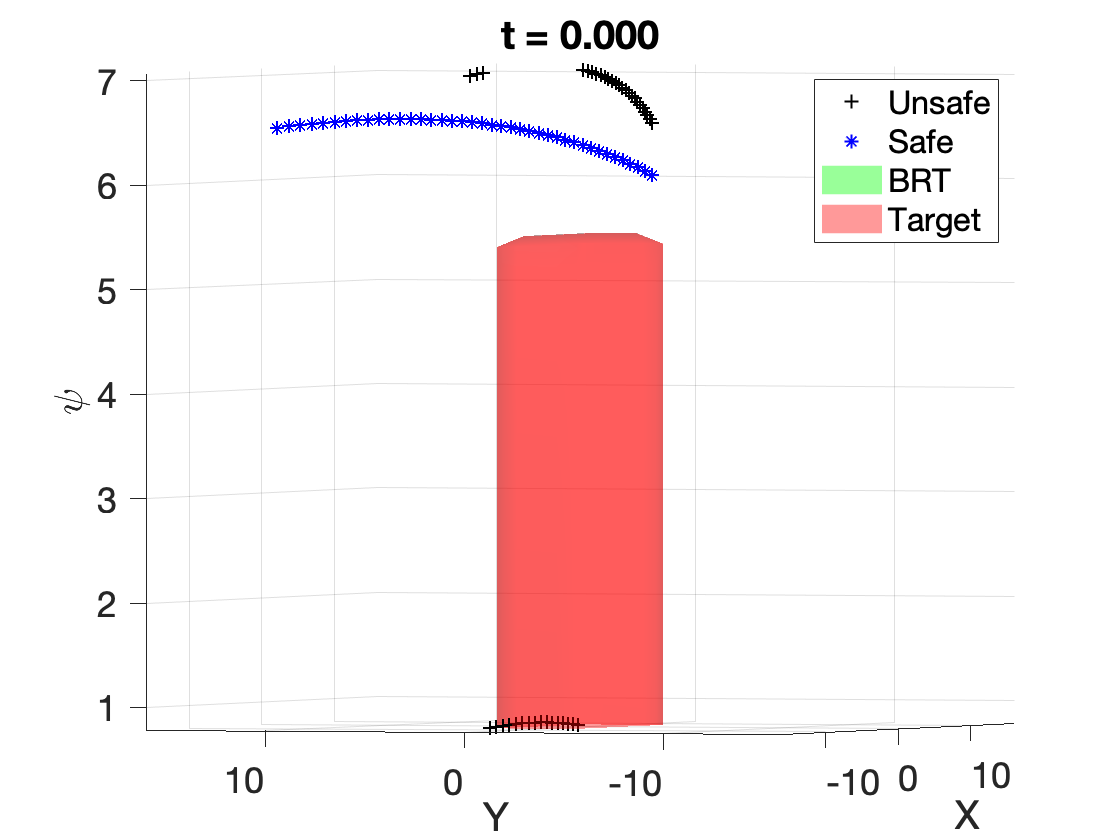}
  \end{subfigure}
  \caption{Comparison of a safe trajectory (blue) and an unsafe trajectory (black) of 6D Bicycle in $(X, Y, \psi)$ space within the time horizon of $2s$. The safe trajectory starts from outside the BRT (green), and successfully avoids all BRTs and the targets (red) within $2s$. The unsafe trajectory starts from inside the BRT and finally hits the target from bottom.}
  \label{fig:safe_traj_1}
  \vspace{-1em}
\end{figure}

In Fig. \ref{fig:safe_traj_1}, the trajectories are in $(X,Y,\psi)$ space. 
The safe initial condition (blue) starts from outside the BRT (green), while the unsafe one (black) starts inside. 
The initial $(v_{x}, v_{y}, \omega)=(-10,1,0.8)$ are the same for both agents.
As time moves forward, the blue trajectory can always stay outside of the BRT at the corresponding time, and avoids the target (red) during the time horizon of two seconds. However, the unsafe trajectory enters the target from the bottom of the plot at $s=-0.8$ (the $\psi$ dimension is periodic).

\section{Conclusion}\label{sec:conclusion}

We propose a decomposition method that largely alleviates the computation complexity for approximating minimal BRTs, without introducing much conservatism. 
Our method is able to analyze many loosely coupled, high-dimensional systems, and
we provide a simple way of making trade-offs between computational requirements and degree of conservatism. 

In the future,
we hope to explore more techniques such as in \cite{lee2019removing} to overcome the constraints for target sets.


\bibliography{references.bib}
\bibliographystyle{IEEEtran}

\end{document}